\documentclass[12pt, paperwidth=5.35in, paperheight=8.3in]{amsart}

\pdfpagewidth=\paperwidth
\pdfpageheight=\paperheight

\newtheorem{theorem}{Theorem}[section]
\newtheorem{lemma}[theorem]{Lemma}

\theoremstyle{definition}

\theoremstyle{remark}
\newtheorem{remark}[theorem]{Remark}
\newtheorem{corollary}[theorem]{Corollary}
\newtheorem{proposition}[theorem]{Proposition}
\numberwithin{equation}{section}

\allowdisplaybreaks
\usepackage{amsmath,amssymb,amscd,amsthm,amsxtra}
\usepackage[colorlinks=true, pdfstartview=FitV, linkcolor=blue, citecolor=blue, urlcolor=blue]{hyperref}
\usepackage{color}
\usepackage{graphicx,epstopdf,verbatim}

\def\BBR {{\mathbb R}}
\def\BBC {{\mathbb C}}

\newcommand {\miw}{  m_Iw }
\newcommand {\mjw}{  m_Jw }

\newcommand {\mlw}{  m_Lw }
\newcommand {\mkw}{  m_Kw }

\newcommand {\mlwi}{ m_L w^{\scriptscriptstyle{-1}}}
\newcommand {\mkwi}{ m_K w^{\scriptscriptstyle{-1}}}

\newcommand {\diw}{\frac{|\Delta_I w|}{m_I w}}

\newcommand {\dkw}{\frac{|\Delta_K w|}{m_K w}}

\newcommand {\dkwi}{\frac{|\Delta_K w^{-1}|}{m_K w^{\scriptscriptstyle{-1}}}}

\newcommand {\diws}{\frac{|\Delta_I w|^2}{(m_I w)^2}}
\newcommand {\dius}{\frac{|\Delta_I u|^2}{(m_I u)^2}}
\newcommand {\diwis}{\frac{|\Delta_I w^{-1}|^2}{(m_I w^{\scriptscriptstyle{-1}})^2}}

\newcommand{ \diwisp} {\frac{|\Delta_I w^{-\frac{1}{p-1}}|^2}{(m_I w^{\scriptscriptstyle{-\frac{1}{p-1}}})^2}}
\newcommand {\dkws}{\frac{|\Delta_K w|^2}{(m_K w)^2}}
\newcommand {\dkwis}{\frac{|\Delta_K w^{-1}|^2}{(m_K w^{\scriptscriptstyle{-1}})^2}}

\newcommand {\wi}{w^{\scriptscriptstyle{-1}}}
\newcommand {\ap}{A^d_p}

\newcommand {\rhp}{RH^d_p}

\begin{document}
\title{{\sc Sharp bounds for $t$-Haar multipliers on $L^2$ }}

\author[O. BEZNOSOVA]{Oleksandra Beznosova}
\address{Department of Mathematics, Baylor University,
One Bear Place \#97328,Waco, TX 76798-7328, USA}
\email{Oleksandra$_-$Beznosova@baylor.edu}
\author[J.C. MORAES]{JEAN CARLO MORAES} \thanks{The first author was supported by  fellowship CAPES/FULBRIGHT, 2918-06/4}
\address{Universidade Federal de Pelotas - UFPEL, Centro de Engenharias - CENG,
Almirante Barroso 1734, sala 16, Pelotas, RS, Brasil}
\email{jmoraes@unm.edu}
\author[ M.C. Pereyra]{MAR\'{I}A CRISTINA PEREYRA}
\address{Department of mathematics and Statistics, 1 University of New Mexico, Albuquerque, NM 87131-001, USA}
\email{crisp@math.unm.edu}
\subjclass[2010]{Primary 42C99 ; Secondary 47B38}
\keywords{ $A_p$-weights, Haar multipliers, complexity.} \maketitle

\setcounter{equation}{0} \setcounter{theorem}{0}


\begin{abstract}
We  show that  if a weight $w\in C^d_{2t}$ and there is
$q >1$ such that $w^{2t}\in A_q^d$, then
the $L^2$-norm of the $t$-Haar multiplier of complexity $(m,n)$ associated to $w$ depends on the
square root of the $C^d_{2t}$-characteristic of $w$ times  the square root $A^d_q$-characteristic of $w^{2t}$ 
times a constant that depends polynomially on the complexity. In particular,
if  $w\in C^d_{2t}\cap A_{\infty}^d$ then $w^{2t}\in A_q^d$ for some $q>1$.
\end{abstract}

\section{Introduction}


 Recently Tuomas Hyt\"onen  settled the \emph{$A_2$-conjecture} \cite{H}:
 for {\em all}  Calder\'on-Zygmund integral singular operators $T$ in $\mathbb{R}^N$, weights $w\in A_p$,
 there is $C_{p,N,T} >0$ such that,
 \[ \|Tf\|_{L^p(w)}\leq C_{p,N,T} [w]_{A_p}^{\max\{1,\frac{1}{p-1}\}}\|f\|_{L^p(w)}.\]
In his proof he developed and used  a representation valid for  \emph{any} 
Calder\'on-Zygmund operator as an average of Haar shift operators of \emph{arbitrary complexity}, paraproducts and their adjoints.
 See \cite{L1,P4} for surveys of the $A_2$-conjecture. An important and hard part of the proof was to obtain bounds for Haar shifts operators that depended
linearly in the $A_2$-characteristic and at most polynomially in the
complexity. 
 
In this paper we show that  if a weight $w\in C^d_{2t}\cap A_{\infty}^d$, then 
the $L^2$-norm of the $t$-Haar multiplier of complexity $(m,n)$ associated to $w$ depends on the
square root of the $C^d_{2t}$-characteristic of $w$ times  the square root
 $A^d_q$-characteristic of $w^{2t}$ for some $q>1$ depending on $t\in\BBR$ 
times a constant that depends polynomially on the complexity.

For $t\in\mathbb{R}$, $m,n\in \mathbb{N}$, and  a weight $w$,  the {\em $t$-Haar multiplier
of complexity $(m,n)$} was introduced in \cite{MoP}, and   is defined formally by
\[ T_{t,w}^{m,n}f(x)= \sum_{L\in\mathcal{D}} \sum_{I\in\mathcal{D}_m(L), J\in\mathcal{D}_n(L)} \frac{\sqrt{|I|\, |J|}}{|L|}\, \frac{w^t(x)}{(m_Lw)^t}\langle f,h_I\rangle h_J(x),\]
 where $\mathcal{D}$ denotes the dyadic intervals, $|I|$  the length of interval $I$,
 $\mathcal{D}_m(L)$ denotes the dyadic subintervals of $L$  of length $2^{-m}|L|$,
  $h_I$ is a Haar function associated to $|I|$, and $\langle f,g\rangle$ denotes the $L^2$-inner product.

When $(m,n)=(0,0)$ we denote the corresponding Haar multiplier by $T^t_w$, and, if in adittion $t=1$, simply $T_w$.
The Haar multipliers $T_w$  are closely related to the resolvent of the dyadic paraproduct \cite{P}, and appeared
in the study of Sobolev spaces on Lipschitz curves \cite{P3}. 

 A necessary condition for the boundedness
of $T_{w,t}^{m,n}$ on $L^2(\BBR^N)$ is that $w\in C^d_{2t}$, see \cite{MoP}, that is,
\[[w]_{C^d_{2t}}:=\sup_{I\in\mathcal{D}} \Big (\frac{1}{|I|}\int_I w^{2t}(x)dx \Big )\Big (\frac{1}{|I|}\int_I w(x)dx\Big )^{-2t}<\infty.\] 
This condition is also sufficient for $t<0$ and $t\geq 1/2$.
Notice that for $0\leq t < 1/2$ the condition $C^d_{2t}$ is always fulfilled; in this case, boundedness of $T^{m,n}_{w,t}$ is known when $w\in A_{\infty}^d$ 
\cite{MoP, KP}. 
 The first author showed in \cite[Chapter 5]{Be}, that if $w\in C^d_{2t}$ and $w^{2t}\in A_q^d$ then the $L^2$-norm of  $T^t_w$,
  is bounded by a constant times ${[w]^{\frac12}_{C_{2t}}[w^{2t}]^{\frac12}_{A_q^d}}$. 
Here we present a different proof of this result
that  holds for  $t$-Haar multipliers of complexity $(m,n)$ with polynomial dependence on the complexity.

\begin{theorem}\label{thm:oleksandra} Let $w\in C^d_{2t}$ and assume there is $q>1$ such that  $w^{2t}\in A_q^d$, then
there is a constant $C_q >0$ depending
only on $q$, such that 
\[ \|T_{t,w}^{m,n}f\|_2\leq C_q(m+n+2)^3[w]_{C^d_{2t}}^{\frac{1}{2}}[w^{2t}]^{\frac{1}{2}}_{A^d_q}\|f\|_2.\]
\end{theorem}
When $w^{2t}\in A_2^d$, this was proved in  \cite{MoP}. 

 Using known properties of weights we can replace the condition 
 $w^{2t}\in A_q^d$, by what may seem to be a more natural condition 
$w\in C_{2t}^d\cap A_{\infty}^d$.
 
\begin{theorem}
Let $w\in C^d_{2t}\cap A_{\infty}^d$, then 
\begin{itemize}

\item[(i)] if $0\leq 2t < 1$, there is $q>1$ such that $w\in A_q^d$, then $w^{2t}\in A_q^d$, and
 \[\|T_{t,w}^{m,n}f\|_2\leq C_q(m+n+2)^3[w^{2t}]^{\frac{1}{2}}_{A^d_q}\|f\|_2\leq C_q(m+n+2)^3[w]^{t}_{A^d_q}\|f\|_2.\]

\item[(ii)] If  $2t \geq 1$ and $w\in A^d_p$ then for $q=2t(p-1)+1$, $w^{2t}\in A_q^d$, and
\[\|T_{t,w}^{m,n}f\|_2\leq C_q(m+n+2)^3[w]_{C^d_{2t}}^{\frac{1}{2}}[w^{2t}]^{\frac{1}{2}}_{A^d_q}\|f\|_2\leq C_p(m+n+2)^3[w]_{C^d_{2t}}[w]_{A^d_p}\|f\|_2.\]

 \item[(iii)] If  $t<0$ then for $q= 1-2t$,  $w^{2t}\in A_q^d$, and the bound becomes linear in the $C_{2t}^d$ characteristic of $w$,
\[\|T_{t,w}^{m,n}f\|_2\leq C (m+n+2)^3[w]_{C^d_{2t}}\|f\|_2.\]

 \end{itemize}

\end{theorem}
The result was known to be  optimal when $t=\pm 1/2$ \cite{Be, P2}.  The bound in (ii)  is not optimal 
since for $t=1$, the $L^2$ norm of $T_w$  is bounded by a constant times $[w]_{C^d_2}D(w)$, where $D(w)$ is  
the doubling constant of $w$,  see \cite{P2}. Here we get the larger norm  $C [w]_{C^d_2}[w]_{A_p^d}$.


To prove this theorem we modify the argument in \cite{MoP} that works when $w\in A_2^d$ ($p=2$). In particular 
we need a couple of new $A_p$-weight lemmas
that are proved using Bellman function techniques: the $A_p$-Little Lemma,  and the $\alpha\beta$-Lemma. 

A few open questions remain. In  case (i) $0<2t<1$, is $w^{2t}\in A_{\infty}^d$ a necessary condition for the
boundedness of $T^{m,n}_{tw}$? Here we show is sufficient.
 Is it possible to get an estimate independent of $q>1$  such that $w^{2t}\in A_q^d$? More specifically, can we replace
$C_q[w^{2t}]_{A_q^d}^{1/2}$ by $C [ w^{2t}]_{A_{\infty}^d}$? or even better by $C D(w)$?
Similarly in case (ii).

The paper is organized as follows. In Section 2 we provide the basic definitions and basic results that are
used through-out this paper. In  Section 3 we  prove the lemmas that are essential for  the main result. 
In Section 4 we  prove the main estimate for the $t$-Haar multipliers with complexity $(m,n)$.
In the Appendix we prove the $A_p$-Little Lemma.

\section{Preliminaries}\label{preliminaries}

\subsection{Weights, maximal function  and dyadic intervals}
 A \emph{weight} $w$ is a locally integrable function in
$\mathbb{R}$  positive almost everywhere.
The $w$-measure of a measurable set $E$, denoted by $w(E)$, is
$ w(E)= \int_E w(x)dx.$
For a measure $\sigma$, $ \sigma(E) = \int_{E} d\sigma$, and $|E|$
stands for the Lebesgue measure of $E$. We define $m^{\sigma}_E f$ to be the
integral average of $f$ on $E$, with respect to $\sigma$,
$$m^{\sigma}_E f := \frac{1}{\sigma(E)} \int_E f(x)
d\sigma.$$
When $dx=d\sigma$ we simply write $m_Ef$, when $d\sigma = v\,dx$ we write $m_E^v f$.

Given a weight $w$, a measurable function $f: \BBR^N\to\BBC$ is 
 in $L^p(w)$ if and only if $\|f\|_{L^p(w)}:= \left (\int_{\mathbb{R}} |f(x)|^pw(x)dx \right )^{1/p}<\infty$.

For a weight $v$ we define the {\em weighted maximal function of $f$} by
$$ (M_vf)(x) = \sup_{I: x\in I} m_I^v |f|$$
where $I$ is a cube in $\BBR^N$ with sides parallel to the axis.
The operator  $M_v$ is bounded in $L^p(v)$ for all $p>1$ and
furthermore
\begin{equation}\label{bddmaxfct}
\|M_v f\|_{L^p(v)} \leq C p' \|f\|_{L^q(v)},
\end{equation}
\noindent where  $p'$ is the dual exponent of
$p$, that is ${1}/{p} + {1}/{p'}=1$. A
proof of this fact can be found in \cite{CrMPz1}. When $v=1$, $M_v$ is the usual   Hardy-Littlewood maximal function, which we will denote by $M$.
It is well-known that $M$ is bounded in $L^p(w)$ if and only if $w\in A_p$ \cite{Mu}.

The collection of all  {\em dyadic intervals}, $\mathcal{D}$,  is given by:
$\mathcal{D}= \cup_{n\in\mathbb{Z}}\mathcal{D}_n$, where $\mathcal{D}_n:=\{I \subset \mathbb{R} \; : \; I=[k2^{-n}, (k+1)2^{-n} ), \; k \in \mathbb{Z}\}.
$
For a dyadic interval $L$ , let $\mathcal{D}(L)$ be  the collection of its dyadic subintervals,
$\mathcal{D}(L):=\{I \subset L \; : \; I \in \mathcal{D}\} ,$
 and let $\mathcal{D}_n(L)$ be the $n^{th}$-generation of dyadic subintervals of $L$,
$\mathcal{D}_n(L):=\{I \in \mathcal{D}(L) \; : \; |I|=2^{-n}|L| \}.$

For every dyadic interval $I \in
\mathcal{D}_n$ there is exactly one $\widehat{I} \in
\mathcal{D}_{n-1}$, such that $I  \subset \widehat {I}$,
$\widehat{I}$ is called the {parent of $I$}. 
Each dyadic interval $I$ in $\mathcal{D}_n$ has
two children in $\mathcal{D}_{n+1}$, the right and left halves, denoted $I_+$ and $I_-$ respectively.

A weight $w$ is {\em dyadic doubling} if  ${w(\widehat{I})}/{w(I)} \leq C\;$ for all $\; I \in \mathcal{D}$.
The smallest constant $C$ is called the doubling constant of $w$ and
 is denoted by $D(w)$. Note that $D(w)\geq 2$, and that in fact the ratio between the length
of a child and the length of its parent is comparable to one, more precisely,
$ D(w)^{-1}\leq {w(I)}/{w(\widehat{I})}\leq 1- D(w)^{-1}$.

\subsection{Dyadic $A^d_p$, reverse H\"older $RH_p^d$ and $C_s^d$ classes}
A  weight $w$ is said to belong to the {\em  dyadic Muckenhoupt $A_p^d$-class} if and only if
\[  [w]_{A_p^d}:= \sup_{I\in \mathcal{D}} (\miw )(m_Iw^{\frac{-1}{p-1}})^{p-1} <\infty,\quad\quad\mbox{for}\quad 1< p<\infty ,\]
where $[w]_{A_p^d}$ is called the $A_p^d$-characteristic of the weight.
If a weight is in $A_p^d$  then it is dyadic doubling. These classes are nested, 
$A_p^d\subset A_q^d$ for all $p\leq q$.
The class $A^d_{\infty}$ is defined by $A^d_{\infty}:= \bigcup_{p>1}A_p^d$. 

A  weight $w$ is said to belong to the {\em dyadic reverse H\"older  $RH_p^d$-class} if and only if
\[ [w]_{RH_p^d}:= \sup_{I\in \mathcal{D}}(m_I w^p)^{\frac{1}{p}}(m_Iw)^{-1}<\infty, \quad \quad\mbox{for}\quad 1<p<\infty,\]
where $[w]_{RH_p^d}$ is called the $RH_p^d$-characteristic of the weight.
If a weight is in $RH_p^d$  then it is not necessarily dyadic doubling
(in the non-dyadic setting reverse H\"older weights are always doubling). Also these classes are nested, 
$RH_p^d\subset RH_q^d$ for all $p\geq q$.
The class $RH^d_1$ is defined by $RH^d_1:= \bigcup_{p>1}RH_p^d$.
 In the non-dyadic setting  $A_{\infty}=RH_1$.
 In the dyadic setting  the collection of dyadic doubling weights in $RH_1^d$ is $A_{\infty}^d$,  hence $A_{\infty}^d$ is
 a proper subset of $RH_1^d$.
See \cite{BeRez} for some recent and very interesting results relating these classes.

\noindent The following are well-known properties of weights (see \cite{JN}) for (ii)):
\begin{lemma}\label{lem:ApRHq}
The following hold
\begin{itemize}
\item If $0\leq s\leq 1$ and $w\in A_{\infty}^d$ then  $w^s\in A_{\infty}$. More precisely, 
if $p>1$ and $w\in A_p^d$ then  $w^s\in A_p$, and $[w^s]_{A_p^d}\leq [w]^s_{A_p^d}$.
\item If $s,q >1$ then $w\in RH_s^d\cap A_q^d$ if and only if $w^s\in A_{s(q-1)+1}$. Moreover
$ [w^s]_{A_{s(q-1)+1}}\leq [w]^s_{RH_s^d}[w]_{A_q^d}^{s}$, $[w]^s_{A_q^d}\leq [w^s]_{A_{s(q-1)+1}}$, and
$[w]_{RH_s^d}^s\leq [w^s]_{A_{s(q-1)+1}}.  $
\item If $p>1$, and $1/p + 1/p'=1$,  then $w\in A_p^d$ if and only if $w^{-1/p-1}\in A_{p'}$.  Moreover
$[w]_{A_p^d}=[w^{-1/p-1}]_{A_{p'}^d}^{p-1}$.
\end{itemize}
\end{lemma}

The following property can be found in \cite{GaRu},
\begin{lemma}\label{lem:doublingRHqAp}
If $w\in RH_s^d\cap A_q^d$  then for all $E\subset B$,
$$  \big ({|E|}/{|B|}\big )^q[w]_{A_q^d}^{-1}\leq {w(E)}/{w(B)} \leq  \big ( {|E|}/{|B|}\big )^{1-\frac{1}{s}}[w]_{RH_s^d}.$$
In particular  $ D(w)\leq 2^q[w]_{A_q^d}.$ 
\end{lemma} 

A weight $w$ satisfies the {\em $C^d_s$-condition}, for $s \in\mathbb{R}$, if
$$[w]_{C^d_s}:= \sup_{I \in \mathcal{D}} \big (m_I w^s\big )\,\big (\miw\big )^{-s} < \infty.$$
The quantity defined above is called the $C^d_s$-characteristic of
$w$. The class of weights  $C^d_s$ was defined in \cite{KP}. Let us analyze this definition.
For $0 \leq s \leq 1$, we have that any weight satisfies the
condition with $C_s^d$-characteristic $1$, this is just a
consequence of H\"older's Inequality (for $s=0,1$ is trivial).
When $s>1$, the condition is analogous to the dyadic reverse
H\"older condition and
$[w]^{{1}/{s}}_{C^d_s} =[w]_{RH^d_s}.$
For $s<0$, we have that  $ w \in C^d_s$ if and only if $w \in A^d_{1- 1/s},$ moreover
$[w]_{C^d_s} = [w]^{-s}_{A^d_{1-1/s}}$.

\begin{lemma}\label{lem:Csvsws}
If  $w\in C_s^d \cap A_{\infty}^d$ then the following hold
\begin{itemize}
\item  For all $0\leq s\leq 1$, there is a $p>1$ such that  $w^s\in A_p$. 
\item  If $s>1$ then there is $q>1$ such that $w^s\in A_{s(q-1)+1}$.
\item If  $s<0$ then $w^s\in A_{1-s}$.
\end{itemize}
\end{lemma}

The proof of this lemma is a direct application of Lemma~\ref{lem:ApRHq} item by item.

\subsection{Weighted Haar functions}
For a given weight $v$ and an interval $I$  define the {\em weighted Haar function} as
\begin{equation}\label{def:Haarfunction}
 h^v_I(x)= \frac{1}{v(I)}\left (  \sqrt{\frac{v(I_-)}{v(I_+)}}\,\chi_{I_+}(x)-\sqrt{\frac{v(I_+)}{v(I_-)}} \,\chi_{I_-}(x)\right ),
 \end{equation}
where $\chi_I(x)$  is the characteristic function of the interval $I$.

If $v$ is the Lebesgue measure on $\mathbb{R}$, we will denote the
{\em Haar function} simply by $h_I$. 
It is a simple exercise to verify that the weighted and unweighted Haar functions are related linearly as follows,
\begin{proposition}\label{whaarbasis}
For any weight $v$, there are numbers $\alpha_I^v$,
$\beta^v_I$ such that
$$ h_I(x) = \alpha^v_I \,h^v_I(x) + \beta_I^v \,{\chi_I(x)}/{\sqrt{|I|}}$$
where
{(i)} $|\alpha^v_I | \leq \sqrt{m_Iv},$
{ (ii)}  $|\beta^v_I| \leq {|\Delta_I v|}/{m_Iv},$
 $\Delta_I v:= m_{I_+}v - m_{I_-}v.$
\end{proposition}

The family $\{h_I^v\}_{I\in\mathcal{D}}$ is an orthonormal  system in $L^2(v)$, with
inner product  $\langle f,g\rangle_v:= \int_{\BBR} f(x)\,\overline{g(x)}\,v(x)dx$.
\subsection{Carleson sequences}

If $v$ is a weight,  a positive sequence $\{\alpha_I\}_{I\in \mathcal{D}}$ is called
a {\em $v$-Carleson sequence with intensity $B$} if for all $J\in \mathcal{D}$,
\begin{equation}\label{def:vCarlesonseq}
\frac{1}{|J|} \sum_{I \in \mathcal{D}(J)}
{\lambda_I}\leq B\; m_Jv.
\end{equation}
When $v=1$ we call a sequence satisfying \eqref{def:vCarlesonseq}  for all ${J \in \mathcal{D}} $ a
{\em Carleson sequence with intensity} $B$. 

\begin{proposition}\label{algcarseq}
Let $v$ be a weight, 
$\{\lambda_I\}_{I\in \mathcal{D}}$ and $\{\gamma_I\}_{I\in \mathcal{D}}$ be two $v$-Carleson sequences with
intensities $A$ and $B$ respectively then for any $c, d >0$ we have
that
\begin{itemize}
\item [(i)]$ \{c \lambda_I + d\gamma_I\}_{I\in \mathcal{D}}$ is a $v$-Carleson sequence with
intensity  $cA + dB$.

\item [ (ii)]$\{ \sqrt{\lambda_I} \sqrt{\gamma_I}\}_{I\in \mathcal{D}}$ is a $v$-Carleson sequence
with intensity  $\sqrt{AB}$.

\item [(iii)]$ \{( c\sqrt{\lambda_I} + d \sqrt{\gamma_I})^{2}\}_{I\in \mathcal{D}}$ is a $v$-Carleson sequence
with intensity  $2c^2A+2d^2B$.

\end{itemize}

\end{proposition}

The proof of these statements is quite simple, see \cite{MoP}.

\section{Main tools}

In this section, we state the lemmas and theorems necessary
to get the  estimate for  the $t$-Haar multipliers
of complexity $(m,n)$.

\subsection{Carleson Lemmas}

The Weighted Carleson Lemma we present here is a variation in the spirit of other weighted Carleson embedding 
theorems that appeared before in the literature \cite{NV, NTV1}. You can find a proof in \cite{MoP}.
%


\begin{lemma}[Weighted Carleson Lemma]\label{weightedCarlesonLem}
Let $v$ be a 
 weight, then $\{\alpha_{L}\}_{L \in \mathcal{D}}$ is a $v$-Carleson sequence
with intensity $B$ if and only if for all   non-negative $v$-measurable functions $F$ on the
line,
\begin{equation}\label{eqn:WCL}
\sum_{L \in \mathcal{D}} \alpha_{L} \inf_{x \in L} F(x)  \leq B
\int_{\mathbb{R}}F(x) \,v(x)\,dx.
\end{equation}
\end{lemma}

The following lemma we view as a finer replacement for H\"older's inequality:
 $ 1\leq (m_I w ) (m_I w^{-1/(p-1)})^{p-1}$.

\begin{lemma}[$A_p$-Little Lemma]\label{litlem}
Let $v$ be a weight, such that $v^{-1/(p-1)}$ is a a weight as well, and
let $\{ \lambda_I \}_{I \in \mathcal{D}}$ be a Carleson sequence with
intensity $Q$ then  $\{{ \lambda_I}/{(m_Iv^{-1/(p-1)})^{p-1}} \}_{I \in \mathcal{D}}$ is  a $v$-Carleson sequence with intensity $4Q$, that is
for all $J\in \mathcal{D}$,
\[ \frac{1}{|J|} \sum_{I \in \mathcal{D}(J)} \frac{\lambda_I}{(m_Iv^{-1/(p-1)})^{p-1}}\leq 4Q \; m_Jv.\]
\end{lemma}

For $p=2$ this was
proved  in \cite[Proposition 3.4]{Be}, or \cite[Proposition 2.1]{Be1},  using the same Bellman function as in the proof
we present in  the Appendix.

\begin{lemma}[\cite{NV}]\label{folklem}
Let $v$ be a weight such that $v^{-1/(p-1)}$ is also a weight.
Let $\{\lambda_{J}\}_{J \in \mathcal{D}}$ be a Carleson sequence
with intensity $B$. Let $F$ be a non-negative measurable function on the
line. Then
\begin{equation*}
\sum_{J \in \mathcal{D}} \frac{ \lambda_{J} }{(m_J v^{-1/(p-1)})^{p-1}} \inf_{x \in J} F(x)
\leq C \;B \int_{\mathbb{R}}F(x)\,v(x)\,dx.
\end{equation*}
\end{lemma}

Lemma~\ref{folklem}  is an immediate consequence of  Lemma~\ref{litlem}, and the
Weighted Carleson Lemma~\ref{weightedCarlesonLem}.
Note that Lemma~\ref{litlem}  can be deduced from Lemma~\ref{folklem}  with $F(x)=\chi_J(x)$.

The following lemma , for $v=w^{-1}$, and 
for $\alpha={1}/{4}$ appeared in \cite{Be}, and for $0<\alpha <1/2$,  in \cite{NV}.
 With small modification in their  proof, using the Bellman function
 $B(x,y)=x^{\alpha}y^{\beta}$ with domain of definition the first quadrant $x,y >0$, we can accomplish the result below, for a complete proof
 see \cite{Mo}.

 \begin{lemma} \label{alphalemma} \emph{($\alpha\beta$-Lemma)}
  Let $u, v$ be  weights. 
 Then for any $J \in \mathcal{D}$ and any $\alpha, \beta \in (0, {1}/{2})$
 \begin{equation}\label{alphameasure}
   \frac{1}{|J|} \sum_{I \in \mathcal{D}(J)}  \dius  |I| (m_Iu)^{\alpha} (m_Iv)^{\beta} \leq C_{\alpha,\beta}(m_Ju)^{\alpha} (m_Jv)^{\beta}.
 \end{equation}
   The constant $C_{\alpha,\beta}={36}/{\min\{\alpha -2\alpha^2,\beta-2\beta^2\}}$.
 \end{lemma}
 From this lemma we   immediately deduce   the following,

\begin{lemma} \label{Apalphalemma}
 Let $ 1<q < \infty$, $w\in A_q^d$, then  $\{ \mu^{q,\alpha}_I \}_{I \in \mathcal{D}}$, where 
 \begin{equation*}
  \mu^{q,\alpha }_I:= (m_Iw)^{\alpha}(m_I w^{\frac{-1}{q-1}})^{\alpha(q-1)}|I|
 \bigg( \frac{|\Delta_I w|^2}{(m_I w)^2} + \frac{|\Delta_I w^{\frac{-1}{q-1}}|^2}{(m_I w^{\frac{-1}{q-1}})^2} \bigg), 
  \end{equation*}
is a Carleson sequence with Carleson intensity at most $C_{\alpha}
[w]_{A_q}^{\alpha} $ for any $\alpha \in \big(0, \max\{1/2,
{1}/{2(q-1)}\}\big)$. Moreover, $\{ \nu^q_I \}_{I \in
\mathcal{D}}$, where
\begin{equation*}
  \nu^q_I:= (m_Iw)(m_I w^{\frac{-1}{q-1}})^{(q-1)}|I| 
\bigg( \frac{|\Delta_I w|^2}{(m_I w)^2} + \frac{|\Delta_I w^{\frac{-1}{q-1}}|^2}{(m_I w^{\frac{-1}{q-1}})^2} \bigg) 
    \end{equation*}
is a Carleson sequence with Carleson intensity at most $C[w]_{A_q}$.
\end{lemma}

\begin{proof}
Set $u=w$, $v=w^{-\frac{1}{q-1}}$, $\beta= \alpha(q-1)$. By
hypothesis $0<\alpha < 1/2$ and also $0< \alpha <
{1}/{2(q-1)}$ which implies that $0<\beta < 1/2$, we can now
use Lemma \ref{alphalemma} to show that $\mu^{q,\alpha}_I$ is a Carleson
sequence with intensity at most $c_{\alpha}[w]_{A^d_q}^{\alpha}$.
For the second statement suffices to notice that $\nu^q_I \leq \mu^{q,\alpha}_I
[w]_{A^d_q}^{1-\alpha}$ for all $I \in \mathcal{D}$, 
for some $\alpha\in\big(0, \max\{1/2,{1}/{2(q-1)}\}\big)$ 
\end{proof}

A proof of this lemma for $q=2$ that works on geometric doubling metric spaces can be found in \cite{NV1, V}.
In those papers $\alpha=1/4$ can be used, and in that case the constant $C_{\alpha}$ can be replaced by $288$.

   


\subsection{Lift Lemma}

Given a dyadic interval $L$, and weights $u,v$, we introduce a family of stopping
time intervals $\mathcal {ST}^m_L$ such that the averages  of the weights over any
stopping time interval $K \in \mathcal{ST}^m_L$ are comparable to the
averages on $L$, and $|K|\geq 2^m|L|$. This construction appeared in \cite{NV} for the case $u=w$, $v=w^{-1}$.
We also present a lemma that lifts $w$-Carleson sequences on intervals to
$w$-Carleson sequences on ``$m$-stopping intervals". This was used in \cite{NV}
for a very specific choice of $m$-stopping time intervals $\mathcal {ST}^m_L$.

\begin{lemma}[Lift Lemma \cite{NV}] \label{liftlem}
Let $u$ and $v$  be  weights, 
$L$ be a dyadic interval and $m,n $ be  fixed positive integers.  Let
$\mathcal{ST}^m_L$ be the
collection of maximal stopping time intervals $K \in
\mathcal{D}(L)$, where the stopping criteria are
 either {\em (i)} $\; |\Delta_Ku|/m_Ku + |\Delta_Kv|/m_Kv \geq {1}/{m+n+2}$,
  or  {\em (ii)} $\,|K| = 2^{-m}|L|$. Then for any stopping interval $K\in \mathcal{ST}^m_L$,
$\, e^{-1}m_Lu  \leq m_Ku \leq e\,m_Lu $, and hence also $\, e^{-1}m_Lv \leq m_Kv \leq e\,m_Lv $.
\end{lemma}

Note that the roles of $m$ and $n$ can be interchanged and we get the family $\mathcal{ST}^n_L$ 
using the same stopping condition (i)  and condition (ii) replaced by $|K|=2^{-n}|L|$.
Notice that $\mathcal{ST}^m_L$ is a partition of $L$ in dyadic subintervals of length at least $2^{-m}|L|$.
The following lemma lifts a $w$-Carleson sequence to $m$-stopping time intervals with comparable intensity.
For the particular $m$-stopping time $\mathcal{ST}^m_L$  given by the stopping criteria (i) and (ii) in Lemma~\ref{liftlem}, 
and $w=1$, this appeared in   \cite{NV}.  
\begin{lemma}\label{corliftlemstop}
For each $L \in \mathcal{D}$ let $\mathcal{ST}^m_L$ be a partition of $L$ in dyadic subintervals
of length at least $2^{-m}|L|$. 
Assume $\{\nu_I \}_{I \in \mathcal{D}} $
is a $w$-Carleson sequence with intensity at most $A$, let $\nu^m_L :=
\sum _{K \in \mathcal{ST}^m_L} \nu_K$,  then $\{\nu_L^m\}_{L \in \mathcal{D}}$ is a $w$-Carleson sequence with
intensity at most $(m+1)A$.
\end{lemma}

For proofs you can see \cite{MoP}.


\subsection{Auxiliary quantities}

For a weight $v$, and  a locally integrable function $\phi $ we define the following quantities,
\begin{align}
 P^m_L \phi & := \sum_{I \in \mathcal{D}_m(L)} \; |\langle \phi, h_I
\rangle| \sqrt{|I|/|L|},\label{def:PmL}\\
S^{v,m}_L \phi &:= \sum_{J \in \mathcal{D}_m(L)} |\langle \phi,h_J^v
\rangle_v| \sqrt{m_Jv}{\sqrt{{|J|}/{|L|}}}, \label{def:SvmLphi}\\
R^{v,m}_L \phi &:= \sum_{J \in \mathcal{D}_m(L)} \frac{|\Delta_J
v|}{m_Jv} m_J(|\phi|v)\;{|J|}/{\sqrt{|L|}},\label{def:RvmLphi}
\end{align}
Let $w\in A_q^d$, $\mathcal{ST}^m_L$ be an $m$-stopping time family of subintervals of $L$, $0<\alpha<\max\{1/2,1/2(q-1)\}$, 
 and  $\{\mu_K^q=\mu_K^{q,\alpha}\}_{K\in\mathcal{D}}$ be the Carleson sequence  with intensity  $C_{\alpha}[w]_{A^d_q}$  defined in Lemma~\ref{Apalphalemma}.
For each  $m>0$, we introduce another  sequence $\{\mu^m_L\}$, which is Carleson by Lemma~\ref{corliftlemstop}:


\[\mu^m_L := \sum_{K \in \mathcal{ST}^m_L} \mu^q_K \quad
 \mbox{with intensity}\quad C_{\alpha}(m+1)[w]_{A^d_q}.\]


 We will use the
following estimates for $S^{v,m}_L \phi$ and $R^{v,m}_L \phi$, where $1<p<2$ will be dictated by the proof of the theorem.
\begin{equation}
S^{v,m}_L \phi  \leq \Big( \sum_{J \in \mathcal{D}_m(L)} |\langle
\phi ,h_J^{v} \rangle_{v}|^2 \Big)^{\frac{1}{2}} (m_Lv)^{\frac{1}{2}},
\label{Sestpar} 
\end{equation}
\begin{equation}
R^{v,m}_L \phi  \leq  C\, C^n_m
(m_L v^{\frac{-1}{q-1}})^{\frac{-(q-1)}{2}}(m_Lv)^{\frac{1}{2}} \inf_{x \in
L} \big (M_{\wi}(|g|^p)(x)\big )^{\frac{1}{p}} \sqrt{\mu^m_L}, \label{Restpar}
\end{equation}
 See \cite{NV} for the proof when $q=2$, slight modification of their argument gives the estimate for $R^{v,m}_L\phi$.
Estimating $P_L^n \phi $ is very simple:
\begin{equation}\label{Ppar}
(P_L^m \phi)^2  
\leq \sum_{I \in \mathcal{D}_m(L)} {|I|}/{|L|} \sum_{I \in \mathcal{D}_m(L)} |\langle \phi ,h_I \rangle|^2 
= \sum_{I \in \mathcal{D}_m(L)} |\langle \phi ,h_I \rangle|^2. 
\end{equation}

%
%
%

\begin{remark}
In \cite{NV1}, Nazarov and Volberg extend the results that they had
in \cite{NV} for Haar shifts to  metric spaces with geometric
doubling. Following the same modifications in the argument made from
\cite{NV} to \cite{NV1}, one could obtain the same result as 
in Theorem~\ref{sufcondhaarmult} on a metric space with geometric doubling,
see \cite{Mo1}.
\end{remark}

\section{Haar Multipliers}

For a weight $w$, $t\in\mathbb{R}$, and $m,n\in \mathbb{N}$,  a  {\em $t$-Haar multiplier of complexity $(m,n)$} is the
operator defined as
\begin{equation}
 T^{m,n}_{t,w} f (x) := \sum_{L \in \mathcal{D}} \sum _{I
\in \mathcal{D}_n(L); J \in \mathcal{D}_m(L)} \frac{\sqrt{|I|\,|J|}}{|L|}
\bigg(\frac{w(x)}{m_L w}\bigg)^t \langle f,h_I\rangle h_J(x).
\end{equation}

In \cite{MoP} it is shown that $w\in C_{2t}^d$ is a necessary condition for boundedness of $T^{m,n}_{w,t}$ in $L^2(\BBR )$.
It is also shown that the $C^d_{2t}$-condition is sufficient  for a $t$-Haar multiplier 
of complexity $(m,n)$ to be bounded in $L^2(\mathbb{R} )$ for most $t$;  this was proved in \cite{KP} for the case $m=n=0$. 
Here we are concerned not only with the boundedness but also with the dependence of the operator norm
on the $C^d_{2t}$-constant . For $T^t_w$ and $t=1, \pm 1/2$
this was studied in \cite{P2}.  The first author \cite{Be} was able to obtain estimates, under the additional condition on the weight
$w^{2t}\in A^d_{q}$ for some $q>1$,  for $T_w^t$ and for all $t\in\mathbb{R}$.  Her 
results were generalized for $T_{w,t}^{m,n}$ for all $t$ when $w^{2t}\in A^d_2$,   see \cite{MoP}. We will show that:

\begin{theorem}\label{sufcondhaarmult}
Let $t$ be a real number and $w$ a weight 
  such that $w^{2t}\in A_q^d$ for some $q>1$ (i.e. $w^{2t}\in A_{\infty}^d$), then
$$\|T^{m,n}_{t,w}f \|_{2} \leq C_q(m+n+2)^3 [w]^{\frac{1}{2}}_{C^d_{2t}}[w^{2t}]^{\frac{1}{2}}_{A^d_q}\|f\|_2.$$
\end{theorem}

Using Lemmas~\ref{lem:ApRHq} and~\ref{lem:Csvsws} we can refine the result as follows, where $C^n_m=n+m+2$.

\begin{theorem}\label{thm:refinement} Let $t\in\BBR$, $w\in C^{2t}$ then
\begin{itemize}
\item[(i)] If $0< 2t <1$ and $w\in A_p^d$ then
$$\|T^{m,n}_{t,w}f \|_{2} \leq C_p(C^n_m)^3 [w^{2t}]^{\frac{1}{2}}_{A^d_p}\|f\|_2\leq C_p(C^n_m)^3  [w]^{t}_{A^d_p}\|f\|_2 .$$
\item[(ii)] If $t>1$ and $w\in A_p^d$ then if $q=2t(p-1)+1$
$$\|T^{m,n}_{t,w}f \|_{2} \leq C_p(C^n_m)^3 [w]^{\frac{1}{2}}_{C^d_{2t}}[w^{2t}]^{\frac{1}{2}}_{A^d_{q}}\|f\|_2\leq C_p(C^n_m)^3 [w]_{C^d_{2t}}[w]_{A_p^d}^t .$$
\item[(iii)] If $t<0$ then 
$$\|T^{m,n}_{t,w}f \|_{2} \leq C(C^n_m)^3 [w]_{C^d_{2t}}\|f\|_2 =   C(C^n_m)^3 [w]^{-2t}_{A^d_{1-1/2t}}\|f\|_2.$$
\end{itemize}
\end{theorem}

\begin{remark}\label{remark:constants}Throughout the proof a constant $C_q$ will be a numerical 
constant depending only on the parameter $q>1$ that may change from line to line.
\end{remark}

\begin{proof}[Proof of Theorem~\ref{thm:refinement}]
By Lemma~\ref{lem:Csvsws} if $w\in C^d_{2t}\cap A_{\infty}^d$ then there is $q>1$ such that $w^{2t}\in A_q^d$, matching
cases perfectly. Now use Theorem~\ref{sufcondhaarmult}.
\end{proof}

\begin{proof}[Proof of Theorem~\ref{sufcondhaarmult}]
Fix $f,\; g \in L^2(\mathbb{R})$.
By duality, it is enough to show that
\begin{equation*}
| \langle  T^{m,n}_{t,w}f,g \rangle | \leq C (m+n+2)^3
[w]^{\frac{1}{2}}_{C^d_{2t}}[w^{2t}]^{\frac{1}{2}}_{A^d_q}\|f\|_{2}
\|g\|_{2}.
\end{equation*}
The inner product on the left-hand-side can be expanded into a double sum, that we now estimate,
$$| \langle  T^{m,n}_{t,w}f,g \rangle | \leq \sum_{L\in\mathcal{D} }\sum_{I\in\mathcal{D}_n(L); J\in\mathcal{D}_m(L)} \frac{\sqrt{|I|\,|J|}}{|L|}
\frac{ |\langle f,h_I\rangle|}{(m_L w)^t}\;| \langle gw^t  , h_J \rangle|. $$


Write $h_J$ as a linear combination of a
weighted Haar function and a characteristic function, $h_J =
\alpha_J h^{w^{2t}}_J + \beta_J {\chi_J}/{\sqrt{|J|}}$, where
$\alpha_J = \alpha^{w^{2t}}_J$, $\beta_J = \beta^{w^{2t}}_J$,
$|\alpha_J|\leq \sqrt{m_Jw^{2t}}$, and $|\beta_J|\leq {|\Delta_J(w^{2t})|}/{m_Jw^{2t}}$.
Now break into two terms to be estimated separately so that,
$$| \langle  T^{m,n}_{t,w}f,g \rangle | \leq  \Sigma_1^{m,n} + \Sigma_2^{m,n},$$
where
\begin{eqnarray*}
\Sigma_1^{m,n} &:=& \sum_{L \in \mathcal{D}} \sum _{I \in \mathcal{D}_n(L); J \in \mathcal{D}_m(L)} \frac{\sqrt{|I|\,|J|}}{|L|}
\frac{\sqrt{m_{J}(w^{2t})}}{(m_L w)^t}  |\langle f,h_I\rangle| \; | \langle gw^{t}  , h^{w^{2t}}_J \rangle |,\\
 \Sigma_2^{m,n} &:=&  \sum_{L \in \mathcal{D}} \sum _{I \in \mathcal{D}_n(L); J \in \mathcal{D}_m(L)}
 \frac{|J| \sqrt{|I|} }{|L|(m_L w)^t} \frac{|\Delta_J (w^{2t})|}{m_{J}(w^{2t})} |\langle f,h_I\rangle| \; m_J( |g| w^t ) .
\end{eqnarray*}

Let $p = 2 - (C_n^m)^{-1}$ (note that $2>p>1$, in fact is getting closer to $2$ as $m$ and $n$ increase),
 and define as in (\ref{def:PmL}), (\ref{def:SvmLphi}) and (\ref{def:RvmLphi}), 
the quantities $P^m_L\phi$, $S_L^{v,n}\phi$ and $R_L^{v,n}\phi $, we will use here the case $v=w^{2t}$, for appropriate $\phi$s
and corresponding estimates. Note that $1<p<2$.

The  sequence $\{\eta_I\}_{I\in\mathcal{D}}$ where  
$$\eta_I := (m_I w^{2t}) \; (m_I
w^{\frac{-2t}{q-1}})^{(q-1)} \Big(
\frac{|\Delta_I(w^{2t})|^2}{|m_I w^{2t}|^2} + \frac{|\Delta_I(w^{{-2t}/{(q-1)}})|^2}{|m_I w^{{-2t}/{(q-1)}}|^2} \Big)|I|,$$
  is  a Carleson sequence with intensity $C_q[w^{2t}]_{A^d_q}$ by Lemma~\ref{Apalphalemma}. The sequence
  $\{\eta_L^m\}_{I\in\mathcal{D}}$ where
$${\eta^m_L := \sum_{I \in \mathcal{ST}_L^m} \eta_I},$$
 and the stopping time $\mathcal{ST}_L^m$ is defined as in Lemma~\ref{liftlem} but with respect to the weights $u=w^{2t}$, $v=w^{{-2t}/{(q-1)}}$,
is a Carleson sequence with  intensity $C_q(m+1)[w^{2t}]_{A^d_q}$ by  Lemma~\ref{corliftlemstop}, .

Observe that on the one hand $\langle gw^t,h^{w^{2t}}_J\rangle = \langle gw^{-t},h^{w^{2t}}_J \rangle_{w^{2t}} $, 
and  on the other $m_J(|g|w^t)=m_J(|gw^{-t}|w^{2t})$.
Therefore,
 $$\Sigma_3^{m,n}=\sum_{L \in \mathcal{D}} {(m_Lw)^{-t}}S^{w^{2t},n}_L
(gw^{-t}) \; P_L^mf ,$$
$$\Sigma_4^{m,n} = \sum_{L \in \mathcal{D}} {(m_Lw)^{-t}}
 R^{w^{2t},n}_L (gw^{-t}) \; P^m_L f.$$

\noindent 
Estimates (\ref{Sestpar})  and (\ref{Restpar}) hold for  $S^{w^{2t},m}_L (gw^{-t})$ and $R^{w^{2t},m}_L (gw^{-t})$  
with $v$ and $\phi$ replaced by $w^{2t}$ and $gw^{-t}$:
\begin{eqnarray*}
S^{w^{2t},n}_L (gw^{-t})
& \leq & (m_L w^{2t})^{\frac{1}{2}}\Big(\sum_{J \in \mathcal{D}_m(L)}
|\langle gw^{-t},h_J^{w^{2t}} \rangle_{w^{2t}} |^2\Big)^{\frac{1}{2}},\\
R^{w^{2t},n}_L (gw^{-t}) & \leq &  
C\,C^n_m (m_Lw^{2t})^{\frac{1}{2}}(m_L w^{\frac{2t}{q-1}})^{\frac{-(q-1)}{2}}F^{\frac12}(x)\sqrt{\eta^m_L},
\end{eqnarray*}
where $F(x)=\inf_{x \in L} \big ( M_{w^{2t}}(|gw^{-t}|^p)(x)\big )^{\frac{2}{p}}$.

\noindent{\bf Estimating $\Sigma_1^{m,n}$:}
Plug in the estimates for $S^{w^{2t},n}_L (gw^{-t})$ and $P_L^mf $,
observe that ${(m_Lw^{2t})^{\frac{1}{2}}}/{(m_L w)^t}\leq [w]_{C^d_{2t}}^{\frac12}$,
use the Cauchy-Schwarz inequality, to get,
\begin{eqnarray*}
\Sigma_1^{m,n} 
& \leq  & \sum_{L \in \mathcal{D}} [w]_{C^d_{2t}}^{\frac12} \Big(\sum_{J \in
\mathcal{D}_n(L)} |\langle
gw^{-t},h_J^{w^{2t}} \rangle_{w^{2t}}|^2 \Big)^{\frac{1}{2}} \Big( \sum_{I \in \mathcal{D}_m(L)} |\langle f,h_I \rangle|^2   \Big)^{\frac{1}{2}} \\
& \leq & [w]_{C^d_{2t}}^{\frac{1}{2}} \|f\|_2\Big(\sum_{L \in
\mathcal{D}} \sum_{J \in \mathcal{D}_n(L) }|\langle
gw^{-t},h_J^{w^{2t}} \rangle_{w^{2t}}|^2\Big)^{\frac{1}{2}}\\
& \leq & [w]_{C^d_{2t}}^{\frac{1}{2}}  \|f\|_{2} \| gw^{-t}\|_{L^2(w^{2t})}
\; = \;   [w]_{C^d_{2t}}^{\frac{1}{2}} \|f\|_{2}\|g\|_{2}.
\end{eqnarray*}

\noindent{\bf Estimating $\Sigma_2^{m,n}$:}
Plug in the estimates for $R^{w^{2t},n}_L (gw^{-t})$ and $P_L^mf $,
where $F(x)=\big ( M_{w^{2t}}(|gw^{-t}|^p)(x)\big )^{2/p}$, use the Cauchy-Schwarz inequality and 
 ${(m_Lw^{2t})^{\frac{1}{2}}}/{(m_L w)^t}\leq [w]_{C^d_{2t}}^{\frac12}$ to  get
\[
\Sigma_2^{m,n}  \leq C\, C^n_m [w]_{C^d_{2t}}^{\frac{1}{2}} \|f\|_2
\Big(\sum_{L \in \mathcal{D}}({\eta_L^m}/{(m_L w^{\frac{-2t}{q-1}})^{q-1}})
\inf_{x \in L} F(x) 
 \Big)^{\frac12}.
\]
Now using Weighted Carleson Lemma \ref{weightedCarlesonLem} with $\alpha_L={\eta_L^m}/{(m_Lw^{\frac{-2t}{q-1}})^{q-1}}$ 
(which by Lemma~\ref{litlem} is a $w^{2t}$-Carleson sequence with intensity no larger than 
$C_q(m+1)[w]_{A^d_q}$, $F(x)=\big ( M_{w^{2t}}|gw^{-t}|^p(x)\big )^{2/p}$, and $v=w^{2t}$,
\[
\Sigma_2^{m,n} \leq  C_q(C^n_m)^2 [w]_{C^d_{2t}}^{\frac{1}{2}} [w^{2t}]^{\frac{1}{2}}_{A^d_q} \|f\|_{2}
   \Big\|M_{w^{2t}}(|gw^{-t}|^p)\Big\|^{\frac{1}{p}}_{L^{\frac{2}{p}}(w^{2t})}.
\]
Using (\ref{bddmaxfct}), that is the boundedness of $M_{w^{2t}}$ in $L^{\frac2p}(w^{2t})$ for $2/p>1$,
\begin{eqnarray*}
\Sigma_2^{m,n} & \leq &C_q( C^n_m)^2 (2/p)' [w]_{C^d_{2t}}^{\frac{1}{2}} [w^{2t}]^{\frac{1}{2}}_{A^d_q} \|f\|_{2}\Big\| |gw^{-t}|^p \Big\|^{\frac{1}{p}}_{L^{\frac{2}{p}}(w^{2t})}\\
&\leq & C_q (C^n_m)^3 [w]_{C^d_{2t}}^{\frac{1}{2}}
[w^{2t}]^{\frac{1}{2}}_{A^d_q} \|f\|_{2}\| g\|_{2},
\end{eqnarray*}
Since $(2/p)'=2/(2-p)=2C^n_m$.
The theorem is proved.
\end{proof}


 \section*{Appendix}


 \begin{proof}[Proof of Lemma~\ref{litlem}]
We will show this inequality using a Bellman function type method.
Consider $B(u,v, l):= u - {1}/({v^{p-1}(1+l)}) $ defined on the domain $\mathbb{D}=
\{(u,v,l) \in \mathbb{R}^3, u >0, v>0, uv^{p-1}>1 \quad \text{and} \quad 0 \leq l \leq 1\}$. Note that
$\mathbb{D}$ is convex.
Note that
\begin{equation}\label{range}
   0 \leq B(u,v,l) \leq u \quad \quad \text{for all} \quad (u,v,l) \in \mathbb{D}
\end{equation}
and
\begin{equation}
 ({\partial B}/{\partial l})(u,v,l) \geq {1}/{4v^{p-1}} \quad \quad \text{for all} \quad (u,v,l) \in \mathbb{D}. \label{1}
\end{equation}
and also $-(du, dv, dl) d^2B (du,dv,dl)^t$ is non-negative because, it equals
\begin{eqnarray*}
 & & \hskip -.6in -(du, dv, dl)\left(
                                                  \begin{array}{ccc}
                                                    0 & 0 & 0 \\
                                                    0 & p(1-p)\frac{v^{-p-1}}{1+l} & (1-p)\frac{v^{-p}}{(l+1)^2}  \\
                                                    0 & (1-p)\frac{v^{-p}}{(l+1)^2} & -2\frac{v^{1-p}}{(l+1)^3} \\
                                                  \end{array}
                                                \right)
                         \left(
                          \begin{array}{c}
                            du \\
                            dv\\
                            dl\\
                          \end{array}
                        \right)
                         \nonumber \\
\quad\quad & & \hskip -.5in = \;  p(p-1)\frac{v^{-p-1}}{1+l}(du)^2 + 2(p-1)\frac{v^{-p}}{(l+1)^2}dudv + 2\frac{v^{1-p}}{(l+1)^3}(dv)^2 \geq 0, \label{2}
\end{eqnarray*}
since all terms are positive for $p>1$.

Now let us show that if $(u_-, v_-, l_-)$  and $(u_+,v_+,l_+) $ are
in $\mathbb{D}$ and we define $(u_0,v_0,l)\in \mathbb{D}$ where $l$ is
in between $l_+$ and $l_-$,  $u_0={(u_-+u_+)}/{2}$, 
$v_0={(v_- + v_+)}/{2}$, and  $l_0=(l_-+l_+)/2$,   then  
\begin{equation*}
    B(u_0,v_0,l) - {\big (B(u_-,v_-,l_-)+B(u_+,v_+, l_+\big )}/{2} \geq {|l-l_0|}/{4v^{p-1}_0}
\end{equation*}

Write for $-1\leq t \leq 1$,
$ u(t) ={[(t+1)u_+ + (1-t)u_-]}/{2}$, $ v(t)= {[(t+1)v_+ + (1-t)v_-]}/{2}$,
 and $ l(t) = {[(t+1)l_+ + (1-t)l_-]}/{2},$.
Define $ b(t):= B(u(t),v(t),l(t))$, then $b(0)=B(u_0,v_0,l_0)$,
$b(1) = B(u_+,v_+, l_+)$, $b(-1)=B(u_-,v_-,l_-)$, ${du}/{dt}={(u_+ -
u_-)}/{2}$, ${dv}/{dt}={(v_+ - v_-)}/{2}$ and ${dl}/{dt}={(l_+ - l_-)}/{2}$. 
If $(u_+,v_+, l_+)$ and
$(u_-,v_-, l_-)$ are in $\mathbb{D}$ then $(u(t),v(t),l(t))$ is also in
$\mathbb{D}$ for all $|t| \leq 1$, since $\mathbb{D}$ is convex. It
is a calculus exercise to show that
\begin{equation}\label{calculus}
b(0) - \frac{b(1)+b(-1)}{2} = \frac{-1}{2} \int^{1}_{-1}(1 -
|t|)b''(t)dt
\end{equation}
Also it is easy to check that
$
    -b''(t) = -\big (\frac{du}{dt}, \frac{dv}{dt}, \frac{dl}{dt} \big ) d^2B  \big (\frac{du}{dt}, \frac{dv}{dt}, \frac{dl}{dt} \big )^t.
$
By the Mean Value Theorem and \eqref{calculus},
\begin{eqnarray*}
    && \hskip -.8in  B(u_0,v_0,l) - \frac{B(u_-,v_-,l_-)+B(u_+,v_+, l_+)}{2}  \\
  \quad      & = & (l-l_0)\frac{\partial B}{\partial l}(u_0, v_0, l') - \frac{1}{2} \int^{1}_{-1}(1 -
|t|)b''(t)dt \geq \frac{l-l_0}{4v^{p-1}_0},
\end{eqnarray*}
where $l'$ is a point between $l$ and $l_0={(l_- + l_+)}/{2}$.

Now we can use the Bellman function argument. Let $u_+ = m_{J_+}w$,
$u_- = m_{J_-}w$, $v_+ = m_{J_+} w^{\frac{-1}{p-1}}$, $v_- =m_{J_-} v^{\frac{-1}{p-1}}$, 
$l_+ = \frac{1}{|J_+|Q}\sum_{I \in \mathcal{D}(J_+)} \lambda_I$ and $l_- = \frac{1}{|J_-|Q}\sum_{I \in \mathcal{D}(J_-)} \lambda_I$.
Thus $(u_-,v_-,l_-), (u_+,v_+,l_+) \in\mathbb{D}$ and 
$u_0= m_J w$, $v_0=m_J w^{\frac{-1}{p-1}}$, and $l_0 = \frac{1}{|J|Q}\sum_{I \in \mathcal{D}(J)} \lambda_I$.
Thus
$$ (u_0,v_0,l_0) - \big({(u_- + u_+)}/{2}, {(v_- + v_+)}/{2}, {(l_- + l_+)}/{2} \big)= \big(0,0, {\lambda_J}/{Q|J|}\big).$$
 Then we can run the usual induction on scale arguments using the properties of the Bellman function,
 \begin{align*}
 |J| m_J w &\geq |J| B(u_0, v_0, l_0) \\
&\geq |J|\frac{ B(u_+,v_+,l_+)}{2}+|J|\frac{ B(u_-,v_-,l_-)}{2} + {\lambda_J}/{4Q \big( m_J w^{\frac{-1}{p-1}}\big)^{p-1}}\\
 & = |J_+|B(u_+,v_+,l_+)+|J_-|B(u_-,v_-,l_-)+{\lambda_J}/{4Q \big(m_J w^{\frac{-1}{p-1}}\big)^{p-1}}
 \end{align*}
Iterating, we get
$$m_J w \geq \frac{1}{4Q|J|}\sum_{I \in \mathcal{D}(J)} \frac{\lambda_I}{(m_Iw^{{-1}/{p-1}})^{p-1}}.$$
\end{proof}


\begin{thebibliography}{09}


\bibitem [Be]{Be} O. Beznosova,
 \emph{Bellman functions, paraproducts, Haar multipliers and weighted inequalities}.
 PhD. Dissertation, University of New Mexico  (2008).

\bibitem [Be1]{Be1} O. Beznosova, 
\emph{Linear bound for the dyadic paraproduct on weighted Lebesgue space $L^2(w)$}.
{J. Func. Anal.} {\bf 255} (2008), 994 --1007.



\bibitem [BeRez]{BeRez} O. Beznosova, A. Reznikov, 
\emph{Sharp estimates involving $A_\infty$ and $LlogL$ constants, and their applications to PDE. } 
ArXiv:1107.1885







%
%
\bibitem [CrMPz1]{CrMPz1} D. Cruz-Uribe, J. M. Martell, C. P\'erez,
\emph{Weights, extrapolation and the theory of Rubio the Francia}.
Birkh\"{a}user, 2011.



%

\bibitem [GaRu]{GaRu}
 J. Garc\'{\i}a Cuerva, J. L. Rubio de Francia,
 \emph{Weighted norm inequalities and related topics.}
 North Holland Math. Studies 116. North Holland, 1985.


\bibitem [H]{H} T. Hyt\"onen, 
\emph{The sharp weighted bound for general Calder\'on-Sygmund operators}.
Ann. Math. {\bf 175} (2012), 1473-1506.

%


\bibitem [JN]{JN} R. Johnson, C.J. Neugebauer, 
\emph{Change of variable results for $A_p$- and reverse H¬older  $RH$-classes}.
 Trans. Amer. Math. Soc. 328 (1991), no. 2, 639Ð666. 

\bibitem [KP]{KP} N. H. Katz, M. C. Pereyra, 
\emph{Haar multipliers, paraproducts and weighted inequalities}. 
Analysis of Divergence, {\bf 10}, 3, (1999), 145-170.



\bibitem  [L1]{L1} M. T. Lacey,
{\em The linear bound in $A_2$ for Calder\'on-Zygmund operators: A survey}.
Submitted to the Proceedings of the J\'ozef Marcinkiewicz Centenary Conference, Poznan, Poland. ArXiv:1011.5784


%

\bibitem [Mo]{Mo} J. C.  Moraes, 
\emph{Weighted estimates for dyadic operators with complexity}.
 PhD Dissertation, University of New Mexico, 2011.

\bibitem [Mo1]{Mo1} J. C.  Moraes, 
\emph{Weighted estimates for dyadic operators with complexity in geometrically doubling spaces}.
In preparation.

\bibitem [MoP]{MoP} J. C.  Moraes, M. C. Pereyra,
\emph{Weighted estimates for dyadic paraproducts and $t$-Haar multipliers with complexity $(m,n)$}.
Submitted to Pub. Mat.

\bibitem [Mu]{Mu} B. Muckenhoupt, 
\emph{Weighted norm inequalities for theHardy--Littlewood maximal function}. 
Trans. Amer. Math. Soc. {\bf 165} (1972), 207--226.

\bibitem[NRezV]{NRezV} F. Nazarov, A. Reznikov, A. Volberg, 
\emph{The proof of $A_2$ conjecture in a geometrically doubling metric space.} 
ArXiv:1106.1342 

\bibitem [NV] {NV} F. Nazarov, A. Volberg,
 \emph{Bellman function, polynomial estimates of weighted dyadic shifts, and $A_2$ conjecture}. 
Preprint (2011).

\bibitem [NV1] {NV1} F. Nazarov, A. Volberg, 
\emph{A simple sharp weighted estimate of the dyadic shifts on metric spaces with geometric doubling}. 
ArXiv: 11044893v2.

%
\bibitem [NTV1] {NTV1} F. Nazarov, S. Treil and A. Volberg,
 \emph{Two weight inequalities for individual Haar multipliers and  other well localized operators.} 
Math. Res. Lett. {\bf 15} (2008), no.3, 583-597.


\bibitem [P1]{P} M. C. Pereyra, 
\emph{On the resolvents of dyadic paraproducts}. 
Rev. Mat. Iberoamericana {\bf 10}, 3, (1994), 627-664.


\bibitem [P2] {P2} M. C. Pereyra, 
\emph{Haar multipliers meet Bellman function}. 
Rev. Mat. Iberoamericana {\bf 25}, 3, (2009), 799-840.

\bibitem [P3] {P3} M. C. Pereyra, 
\emph{Sobolev spaces on Lipschitz curves}.
Pacific J. Math. {\bf 172} (1996), no. 2, 553--589.

\bibitem [P4]{P4} M. C. Pereyra,
\emph{Dyadic harmonic analysis and weighted inequalities.}
Chapter in ``Excursions in Harmonic Analysis: The February Fourier Talks at the Norbert Wiener Center", 
Edited by T. Andrews, R. Balan, W. Czaja, K. Okoudjou, J. Benedetto.
 Springer 2012.

%

%
%
%
%




\bibitem[V]{V} A. Volberg, 
\emph{Bellman function technique in Harmonic Analysis.
 Lectures of INRIA Summer School in Antibes, June 2011.}
 Preprint (2011) available at 	arXiv:1106.3899

\bibitem [W] {W} J. Wittwer,
\emph{A sharp estimate on the norm of the martingale transform}.
Math. Res. Letters, {\bf 7} (2000), 1--12.

%
\end{thebibliography}
\end{document}